\documentclass[12pt,oneside]{amsart}
\usepackage[margin=1in]{geometry}
\usepackage{amsmath}
\usepackage{amsthm}
\usepackage{amsfonts}
\usepackage{amssymb}
\usepackage[hidelinks]{hyperref}
\usepackage{bm}

\numberwithin{equation}{section}
\numberwithin{figure}{section}

\theoremstyle{plain}
\newtheorem{thm}[equation]{Theorem}
\newtheorem{lemma}[equation]{Lemma}

\newtheorem{prop}[equation]{Proposition}

\theoremstyle{definition}

\newcommand{\Z}{\ensuremath \mathbb{Z}}
\newcommand{\Mod}[1]{\ (\mathrm{mod}\ #1)}
\DeclareMathOperator{\division}{|}
\DeclareMathOperator{\doubledivision}{||}

\begin{document}

\title[Prime factors of $\Phi_3(x)$ of the same form]{Prime factors of $\Phi_3(x)$ of the same form}

\author{Cody S.\ Hansen}
\address{Department of Mathematics, Brigham Young University, Provo, UT 84602, USA}
\email{codyshansen@yahoo.com}

\author{Pace P.\ Nielsen}
\address{Department of Mathematics, Brigham Young University, Provo, UT 84602, USA}
\email{pace@math.byu.edu}

\keywords{cyclotomic polynomial, Eisenstein integers, triple threat}
\subjclass[2020]{Primary 11A51, Secondary 11D09, 11N32, 11R04}

\begin{abstract}
We parameterize solutions to the equality $\Phi_3(x)=\Phi_3(a_1)\Phi_3(a_2)\cdots\Phi_3(a_n)$ when each $\Phi_3(a_i)$ is prime.  Our focus is on the special cases when $n=2,3,4$, as this analysis simplifies and extends bounds on the total number of prime factors of an odd perfect number.
\end{abstract}

\maketitle

\section{Introduction}\label{Section:Introduction}

In this paper we study prime factors of the third cyclotomic polynomial,
\[
\Phi_3(x)=\frac{x^3-1}{x-1}=x^2+x+1,
\]
with $x\geq 1$ an integer.  Congruence conditions on the divisors are well-known and elementary.  Any prime divisor $p$ satisfies $p\equiv 0,1\Mod{3}$.  We have that
\[
3\division\Phi_3(x)\, \text{ if and only if }\, x\equiv 1\Mod{3},
\]
and in this case $3\doubledivision\Phi_3(x)$.  On the other hand, a prime $p\equiv 1\Mod{3}$ is a divisor of $\Phi_3(x)$ if and only if $x$ has order $3$ modulo $p$.  Similar statements hold for arbitrary cyclotomic polynomials, and such results often appear in introductory textbooks in number theory (for example, see Theorem 95 on page 166 of \cite{Nagell}).

These congruence conditions are independent of one another.  Therefore, for any choice of primes $p_1,p_2,\ldots,p_n$ congruent to $1\Mod{3}$, and any choice of exponents $e_1,e_2,\ldots, e_n\geq 1$, we can find some integer $x\geq 1$ such that $p_i^{e_i}\doubledivision\Phi_3(x)$, for each $1\leq i\leq n$.  If these primes and exponents are chosen randomly, any corresponding $x$ is expected to be very large, and so we also expect to find some prime divisor of $\Phi_3(x)$ not equal to any of the $p_i$'s.

On the other hand, according to a conjecture of Bunyakovsky, $\Phi_3(x)$ is prime for arbitrarily large values of $x$.  It is easy to check, numerically, that indeed this polynomial is prime quite often.  An asymptotic for the number of such primes (up to a given size) is provided by the Bateman-Horn conjecture.  At present, no univariate polynomial of degree at least two has been proven to have infinitely many prime values on integer inputs.

This article focuses on the situation when all of the prime factors of $\Phi_3(x)$ are of the same form, or in other words
\[
x^2+x+1=(a_1^2+a_1+1)(a_2^2+a_2+1)\cdots (a_n^2+a_n+1)
\]
where each $a_i^2+a_i+1$ is prime.  When $n\geq 2$ is fixed, our methods demonstrate that any such factorization arises from a finite number of parameterizations, where both $a_n$ and $x$ are expressed as rational functions evaluated in the variables $a_1,\ldots, a_{n-1}$.

Controlling the number and size of the prime divisors of $\Phi_3(x)$ of the same form is important in applications.  For instance, many papers on odd perfect numbers---especially those involving numerical searches---need such information; see \cite{FNO,Nielsen10,OchemRao}.  In particular, the recent work of Zelinsky \cite{Zelinsky} on this topic is what motivated the work in this paper.  He raised the question of whether it is possible to have what he termed a \emph{triple threat}, which is a solution to
\[
\Phi_3(x)=\Phi_3(a)\Phi_3(b)\Phi_3(c),
\]
with the seven quantities $x,a,b,c,\Phi_3(a),\Phi_3(b)$, and $\Phi_3(c)$ all prime.  Our work shows that this and other similar situations are not possible.  Consequently, our work simplifies the proofs---as well as extends the bounds achieved---in the paper \cite{Zelinsky}.

\section{Factorization in the Eisenstein integers}

The results of this paper depend heavily on knowledge of factorization in the ring of Eisenstein integers, $R=\Z[\zeta_3]$.  We will review the needed facts here.

First, rather than use the $\Z$-basis $\{1,\zeta_3\}$ for $R$, we will find it more convenient to use the basis $\{1,\zeta_6\}$.  Here, as usual, $\zeta_6=\frac{1+\sqrt{-3}}{2}$.  The norm with respect to this basis takes the form
\[
N(x+y\zeta_6)=(x+y\zeta_6)(x+y\overline{\zeta_6})=x^2+xy+y^2.
\]
In particular, $N(x+\zeta_6)=x^2+x+1$.  Thus, understanding the integer prime factorization of $x^2+x+1$ corresponds to understanding the prime factorization of $x+\zeta_6$ in $R$. Note that it is well-known that $R$ is a Euclidean domain, and hence a UFD, so we can freely speak about prime factorizations in $R$.

The unit group of $R$ is generated by $\zeta_6$.  As $N(\zeta_6)=1$, the norm of every unit of $R$ is $1$.  For any integer $x\geq 0$, the unit multiples of $x+\zeta_6$ and of $\overline{x+\zeta_6}$, when again written with respect to the basis $\{1,\zeta_6\}$, are as follows:
\[
\begin{array}{rclrcl}
\zeta_6^0\cdot (x+\zeta_6) & = & x+\zeta_6, & \qquad \zeta_6^0\cdot (\overline{x+\zeta_6}) & = & (x+1)-\zeta_6, \\[2pt]
\zeta_6^1\cdot (x+\zeta_6) & = & -1+(x+1)\zeta_6, & \qquad \zeta_6^1\cdot (\overline{x+\zeta_6}) & = & 1+x\zeta_6, \\[2pt]
\zeta_6^2\cdot (x+\zeta_6) & = & -(x+1)+x\zeta_6, & \qquad \zeta_6^2\cdot (\overline{x+\zeta_6}) & = & -x+(x+1)\zeta_6, \\[2pt]
\zeta_6^3\cdot (x+\zeta_6) & = & -x-\zeta_6, & \qquad \zeta_6^3\cdot (\overline{x+\zeta_6}) & = & -(x+1)+\zeta_6, \\[2pt]
\zeta_6^4\cdot (x+\zeta_6) & = & 1-(x+1)\zeta_6, & \qquad \zeta_6^4\cdot (\overline{x+\zeta_6}) & = & -1-x\zeta_6, \\[2pt]
\zeta_6^5\cdot (x+\zeta_6) & = & (x+1)-x\zeta_6, & \qquad \zeta_6^5\cdot (\overline{x+\zeta_6}) & = & x-(x+1)\zeta_6.
\end{array}
\]
If $x=0$, then the left column cycles through the units, and the right column repeats the left column but shifted two entries.  If $x=1$, then the left column cycles through the unit multiples of the ramified prime $1+\zeta_6$, and the right column again repeats the left column but shifted one entry.  If $x\geq 2$, there are no repetitions among the twelve entries.

For all twelve entries, if neither of the two coefficients (with respect to the basis $\{1,\zeta_6\}$) is $\pm 1$, then the two coefficients add to $\pm 1$.  This leads us to the following fundamental fact:

\begin{lemma}\label{Lemma:Fundamental}
Given an element $m+n\zeta_6\in R$, with $m,n\in \Z$, the following are equivalent:
\begin{itemize}
\item[\textup{(1)}] It holds that $m=\pm 1$, or $n=\pm 1$, or $m+n=\pm 1$.
\item[\textup{(2)}] There exists some \textup{(}unique\textup{)} integer $x\geq 0$, such that $m+n\zeta_6$ is a unit multiple of either $x+\zeta_6$ or its complex conjugate.
\end{itemize}
Moreover, if we know which of the six cases occurs in \textup{(1)}, and we know the signs of both $m$ and $n$, then we can describe $x$ as a linear polynomial in $m$ and $n$.
\end{lemma}
\begin{proof}
We already observed, looking at the twelve entries above, that $(2)\Rightarrow (1)$.  Conversely, suppose that $(1)$ holds.  When $m=1$, then if $n\geq 0$ we can take $x=n$, while if $n<0$ we can take $x=-n-1$.  The other five cases work out similarly.  Finally, the uniqueness of $x$ comes from the fact that
\[
N(m+n\zeta_6)=N(x+\zeta_6)=x^2+x+1,
\]
which is a strictly increasing function of $x$ when $x\geq 0$.
\end{proof}

The behavior of prime factorization in $R$ is well-known.  For our work, we will only need the following basic facts.  Given an integer $a\geq 1$, if $a^2+a+1$ is prime in $\Z$, then it factors into two conjugate primes
\[
a^2+a+1=(a+\zeta_6)(a+\zeta_6^{-1})=\zeta_6^{-1}(a+\zeta_6)(1+a\zeta_6).
\]
The prime is repeated only in the case when $a=1$; indeed, $3$ is the only prime in $\Z$ that ramifies in $R$.

Suppose now that we have an equality
\[
x^2+x+1=\prod_{i=1}^n(a_i^2+a_i+1)
\]
for some integers $x,a_1,\ldots, a_n\geq 1$, where each $a_i^2+a_i+1$ is an integer prime.  Thus, we have
\[
\zeta_6^n(x+\zeta_6)(\overline{x+\zeta_6})=\prod_{i=1}^{n}(a_i+\zeta_6)(1+a_i \zeta_6),
\]
where the right side is a factorization into primes of $R$.  We then know that $x+\zeta_6$ is (up to a unit) a product of $n$ primes, where exactly one of the two prime factors of $a_i^2+a_i+1$ appears.  Of course, $\overline{x+\zeta_6}$ is the product of the remaining, conjugate prime factors.  There are finitely many possibilities for whether $a_i+\zeta_6$ or $1+a_i\zeta_6$ appears in the factorization of $x+\zeta_6$.  Running through all of these possibilities, and using Lemma \ref{Lemma:Fundamental}, we are able to completely characterize solutions to this equality.  In the next few sections we will fully demonstrate this process in the cases when $n=2$ and $n=3$, and sketch it when $n=4$.

\section{Two factors}

The solutions to $\Phi_3(x)=\Phi_3(a)\Phi_3(b)$, with both $\Phi_3(a)$ and $\Phi_3(b)$ prime, belong to a single infinite family, as described by the following theorem.

\begin{thm}\label{Thm:TwoFactors}
Let $x,a,b\geq 1$ be integers satisfying
\begin{equation}\label{Eq:2Factors}
x^2+x+1=(a^2+a+1)(b^2+b+1),
\end{equation}
where $a\leq b$ and both $a^2+a+1$ and $b^2+b+1$ are primes.  Then, up to reordering the variables, the solutions $(a,b,x)$ belong to the infinite family
\[
(a,a+1,(a+1)^2).
\]
\end{thm}
\begin{proof}
We know that, up to a unit multiple, $x+\zeta_6$ is one of four quantities: $(a+\zeta_6)(b+\zeta_6)$, $(a+\zeta_6)(1+b\zeta_6)$, $(1+a\zeta_6)(b+\zeta_6)$, or $(1+a\zeta_6)(1+b\zeta_6)$.  Thus, after passing to the complex conjugate if necessary, we reduce to the first two cases.
\bigskip

\noindent{\bf Case 1}: Suppose $(a+\zeta_6)(b+\zeta_6)$ is equal to $x+\zeta_6$ or its conjugate, up to a unit.  We compute
\[
(a+\zeta_6)(b+\zeta_6)=(ab-1)+(a+b+1)\zeta_6.
\]
By Lemma \ref{Lemma:Fundamental}, we must have $ab-1=\pm 1$, or $a+b+1=\pm 1$, or $ab+a+b=\pm 1$.  Since $1\leq a\leq b$, the only option not immediately ruled out is when $ab-1=1$.  Thus, $ab=2$, and so $a=1$ and $b=2$.   (Note that Lemma \ref{Lemma:Fundamental} asserts that $x$ is always uniquely determined by $a$ and $b$, and in this case $x=4$.)  This is a special case of the general solution stated in the theorem.
\bigskip

\noindent{\bf Case 2}: Suppose $(a+\zeta_6)(1+b\zeta_6)$ is equal to $x+\zeta_6$ or its conjugate, up to a unit.  We compute
\[
(a+\zeta_6)(1+b\zeta_6)=(a-b)+(ab+b+1)\zeta_6.
\]
By Lemma \ref{Lemma:Fundamental}, we must have $a-b=\pm 1$, or $ab+b+1=\pm 1$, or $ab+a+1=\pm 1$.  Again, since $1\leq a\leq b$, the only option not immediately ruled out is when $a-b=-1$.  Thus, $b=a+1$, and one can directly check that $x=(a+1)^2$ is the unique solution to \eqref{Eq:2Factors}.
\end{proof}

Numerical searches suggest that there are indeed infinitely many cases where $\Phi_3(a)$ and $\Phi_3(a+1)$ are simultaneously prime.  This also would follow from standard conjectures in number theory, such as Schinzel's hypothesis H.

On the other hand, $x$ is never prime in any of these solutions.  Thus, as was already known, there are no ``double threats''.

\section{Three factors}

The solutions to $\Phi_3(x)=\Phi_3(a)\Phi_3(b)\Phi_3(c)$, with $\Phi_3(a)$, $\Phi_3(b)$, and $\Phi_3(c)$ simultaneously prime, are slightly more complicated, with three sporadic solutions and one infinite family of solutions.  These are described by the following theorem.

\begin{thm}\label{Thm:ThreeFactors}
Let $x,a,b,c\geq 1$ be integers satisfying
\begin{equation}\label{Eq:3Factors}
x^2+x+1=(a^2+a+1)(b^2+b+1)(c^2+c+1),
\end{equation}
where $a^2+a+1$, $b^2+b+1$, and $c^2+c+1$ are primes.  Then, up to reordering variables, the solutions $(a,b,c,x)$ are either
\begin{itemize}
\item one of the three sporadic solutions $(2,2,2,18)$, $(1,2,5,25)$, $(1,3,3,22)$, or
\item belong to the infinite family
\[
\left(a,b,\frac{ab}{a+b+1},\frac{ab}{a+b+1}(ab+a+b)+a+b\right)\ \text{ such that $a\leq b$}.
\]
\end{itemize}
\end{thm}
\begin{proof}
The prime factors of $x+\zeta_6$ are, up to units and up to conjugates and up to permuting variables, one of two cases: $(a+\zeta_6)(b+\zeta_6)(c+\zeta_6)$ or $(a+\zeta_6)(b+\zeta_6)(1+c\zeta_6)$.
\bigskip

\noindent{\bf Case 1}: We compute
\[
(a+\zeta_6)(b+\zeta_6)(c+\zeta_6)=(abc-a-b-c-1) + (ab+ac+bc+a+b+c)\zeta_6.
\]
Without loss of generality we may assume $1\leq a\leq b\leq c$.  Since $x^2+x+1$ is never divisible by $9$, we also know $b\geq 2$.  By Lemma \ref{Lemma:Fundamental}, we must have $abc-a-b-c-1=\pm 1$, or $ab+ac+bc+a+b+c=\pm 1$, or $abc+ab+ac+bc-1=\pm 1$.  The only options not immediately ruled out are $abc-a-b-c-1=\pm 1$.

{\bf Case 1a}: Suppose $abc-a-b-c-1=1$.  Solving for $c$, we get $c=\frac{a+b+2}{ab-1}$.  (Note that $ab-1\neq 0$, since $b\geq 2$.)  For $c$ to be an integer, we need $ab-1\leq a+b+2$.  Solving this inequality in terms of $b$, we get $b\leq \frac{a+3}{a-1}$ (which is valid only when $a\neq 1$).  As $a\leq b$, this means $a\leq \frac{a+3}{a-1}$, or in other words $a^2-2a-3\leq 0$.  This means that $a\in \{1,2,3\}$.

First, if $a=3$, then $b\leq \frac{a+3}{a-1}=3$, hence $b=3$.  But then $c=\frac{a+b+2}{ab-1}=1$, which contradicts the fact that $b\leq c$.

Next, if $a=2$, then $b\leq \frac{a+3}{a-1}=5$, so $b\in \{2,3,4,5\}$.  The only option where $c=\frac{a+b+2}{ab-1}$ is an integer at least as big as $b$ is when $b=2$ and $c=2$.  This is the first listed solution.

Finally, consider when $a=1$. Then $b\leq c=\frac{b+3}{b-1}$.  Thus, from a computation above, we have $b\in \{2,3\}$. If $b=2$ then $c=5$, which is the second listed solution.  Finally, if $b=3$ then $c=3$, which is the third listed solution.

{\bf Case 1b}: Suppose $abc-a-b-c-1=-1$.  Solving for $c$, we get $c=\frac{a+b}{ab-1}$.  For this to be an integer, we need $ab-1\leq a+b$.  Solving in terms of $b$, we get $b\leq \frac{a+1}{a-1}$ (unless $a=1$).  As $a\leq b$, this yields $a\leq \frac{a+1}{a-1}$, or in other words $a^2-2a-1\leq 0$.  Thus $a\in \{1,2\}$.

First, if $a=2$, then $b\leq 3$, and so $b\in \{2,3\}$.  Neither option gives an integer value for $c$ that is at least as big as $b$.

Finally, if $a=1$, then $c=\frac{b+1}{b-1}$. But since $2\leq b\leq c$, this leads to $b=2$ and $c=3$.  This solution is part of the infinite family (after permuting the variables).
\bigskip

\noindent{\bf Case 2}: We compute
\[
(a+\zeta_6)(b+\zeta_6)(1+c\zeta_6)=(ab-ac-bc-c-1)+(abc+ac+bc+a+b+1)\zeta_6.
\]
We again apply Lemma \ref{Lemma:Fundamental}.  The right coefficient is not $\pm 1$, nor is the sum of the two coefficients $\pm 1$.  Thus, the only options are $ab-ac-bc-c-1=\pm 1$.

{\bf Case 2a}: Suppose $ab-ac-bc-c-1=1$.  Solving for $c$, we get $c=\frac{ab-2}{a+b+1}$.  First, note that if $a=1$ then $c=\frac{b-2}{b+2}$, which is never an integer.  Thus, we may assume $1<a$, and by symmetry $1<b$.  For $a^2+a+1$ to be prime, we then must have $a\equiv 0,2\pmod{3}$, and similarly $b\equiv 0,2\pmod{3}$.  We consider each of these possibilities in turn.

If both $a$ and $b$ are congruent to $0\pmod{3}$, then $c\equiv 1\pmod{3}$.  In that case, the only way for $c^2+c+1$ to be prime is if $c=1$.  Thus, $\frac{ab-2}{a+b+1}=1$.  Solving for $b$ we get $b=\frac{a+3}{a-1}$.  The only positive integer value of $a\equiv 0\pmod{3}$ that makes $b$ an integer is $a=3$.  Hence $b=3$, which again gives us the third sporadic solution.

Next, if both $a$ and $b$ are congruent to $2\pmod{3}$, then $c\equiv 1\pmod{3}$.  Thus, once again we get $c=1$ and $b=\frac{a+3}{a-1}$.  The only positive integer value of $a\equiv 2\pmod{3}$ that makes $b$ an integer is when $a=2$.  Hence $b=5$, and this is the second sporadic solution.

Finally, without loss of generality, suppose $a\equiv 0\pmod{3}$ and $b\equiv 2\pmod{3}$.  We find that $3$-adic valuation of $c$ is negative in this case, so it cannot be an integer.

{\bf Case 2b}: Suppose $ab-ac-bc-c-1=-1$.  Solving for $c$, we get $c=\frac{ab}{a+b+1}$.  Solving for $x$ using \eqref{Eq:3Factors} yields the infinite family.
\end{proof}

A computer algebra system (or some elbow grease) can quickly show that the infinite family does indeed satisfy \eqref{Eq:3Factors}. Numerical searches suggest that there are in fact infinitely many integers $a$ and $b$, where $c=\frac{ab}{a+b+1}$ is also an integer, and simultaneously each of $\Phi_3(a)$, $\Phi_3(b)$ and $\Phi_3(c)$ are prime in $\Z$.

It still remains to fulfil our promise from the introduction, in showing that there are no ``triple threats''.

\begin{prop}
It is impossible to have $\Phi_3(x)=\Phi_3(a)\Phi_3(b)\Phi_3(c)$, with each of the seven numbers $x$, $a$, $b$, $c$, $\Phi_3(a)$, $\Phi_3(b)$, and $\Phi_3(c)$ simultaneously prime.
\end{prop}
\begin{proof}
The three sporadic solutions cause no problems, since $x$ is not prime in those cases.  For the infinite family, let us show that $c=\frac{ab}{a+b+1}$ is never prime when $a$ and $b$ are.  We have $c=\frac{a}{a+b+1}b < b$.  Similarly, $c=\frac{b}{a+b+1}a<a$.  However, the only prime factors in the numerator of $c$ are $a$ and $b$.  Thus, $c$ cannot be prime.
\end{proof}

\section{Four (or more) factors}

A significant portion of the proofs of Theorems \ref{Thm:TwoFactors} and \ref{Thm:ThreeFactors} involve nothing more than case analysis and simple inequalities.  These can easily be handled by a modern computer algebra system, and the results can then, \emph{a fortiori}, be checked by hand.  For instance, if we are interested in the situation with four factors, then blindly using Mathematica's ``reduce'' routine gives us the following options:

\begin{prop}
Let $x,a,b,c,d\geq 1$ be integers satisfying
\begin{equation}\label{Eq:FourFactors}
x^2+x+1=(a^2+a+1)(b^2+b+1)(c^2+c+1)(d^2+d+1),
\end{equation}
where the four factors on the right side are each prime.  Then, up to reordering the variables \textup{(}and suppressing $x$\textup{)}, the solutions $(a,b,c,d)$ are either
\begin{itemize}
\item one of the two sporadic solutions $(2,2,2,17)$, $(2,2,3,6)$, or
\item belong to at least one of the four infinite families
\[
\begin{array}{ll}
\displaystyle\left(a,b,c,\frac{abc-a-b-c-2}{ab+ac+bc+a+b+c} \right) & \text{such that $a\leq b\leq c$,}\\[15pt]
\displaystyle\left(a,b,c,\frac{abc-a-b-c}{ab+ac+bc+a+b+c}\right) & \text{such that $a\leq b\leq c$,}\\[15pt]
\displaystyle\left(a,b,c,\frac{abc+ab+a+b-c-1}{ac+bc-ab+c+1}\right)& \text{such that $a\leq b$ and $a\leq c$, or}\\[15pt]
\displaystyle\left(a,b,c,\frac{abc+ab+a+b-c+1}{ac+bc-ab+c+1}\right)& \text{such that $a\leq b$ and $a\leq c$.}
\end{array}
\]
\end{itemize}
\end{prop}

The case when one of the entries of the quadruple $(a,b,c,d)$ equals $1$ is important for applications in \cite{Zelinsky}.  Mathematica tells us (and it is easy to verify) that there are only finitely such tuples, namely (up to rearranging the entries) the following eight:
\begin{eqnarray*}
(1,2,2,5),\ (1,2,2,6),\ (1,2,3,15),\ (1,2,3,17),\\
(1,2,5,24),\ (1,2,6,14),\ (1,2,6,15),\ (1,3,3,21).
\end{eqnarray*}
We note that the only tuple in this list without an even entry is $(1,3,3,21)$, and in that case $x=484$, which is even.

Numerical searches suggest that there exist infinitely many quadruples $(a,b,c,d)$ in each of the four infinite families such that $\Phi_3(a)$, $\Phi_3(b)$, $\Phi_3(c)$, and $\Phi_3(d)$ are simultaneously prime.  We might ask if there exists a ``quadruple threat'', where additionally $a$, $b$, $c$, $d$, and (the corresponding) $x$ are also prime.  The answer is yes; namely, $(2,3,3,5)$, which has the corresponding $x$ value of $191$.  This is the unique example when any of the primes is even.

The paper \cite{Zelinsky} is mainly concerned about \emph{odd} tuples.  Notice that if we make this additional restriction, then both the third family and the fourth family are disqualified; we see that $d$ must be even if $a$, $b$, and $c$ are odd, in those families.  The second family is also disqualified for the following reason.  First, we can check that there are only finitely many solutions where one of the entries is $3$, and none of those solutions pans out. This forces the congruence conditions $a,b,c\equiv 2\pmod{3}$, in order for $\Phi_3(a)$, $\Phi_3(b)$, and $\Phi_3(c)$ to be prime.  In that case, $d$ has a negative $3$-adic valuation, so it is not an integer.

Perhaps surprisingly, the first infinite family is not disqualified, and a directed search finds that taking $x$ to be
\[
91939084808732106267276347132638226863579171653778358025126018412980773335493
\]
we have the odd quadruple threat
\begin{eqnarray*}
\Phi_3(x) & = & \Phi_3(39640921169)\Phi_3(39640924811)\Phi_3(431466989439524477)\\
& &\Phi_3(135601684951723299939542158557248883821).
\end{eqnarray*}
(The primality of the nine quantities was verified using Mathematica's ``ProvablePrimeQ'' routine.)  This was the only odd quadruple threat we found, but we expect that there are infinitely many more examples, and possibly some of smaller size.

A brief computer search shows that if any number $x$ gives rise to an odd quadruple threat, then the smallest prime factor of $\Phi_3(x)$ is bigger than $10^{13}$; this bound could easily be improved with further computations.

\section*{Acknowledgements}

We thank Jeremy Rouse for comments that helped us refine our search for odd quadruple threats.  The first author is planning to apply these results to odd perfect numbers in a forthcoming work.

\bibliographystyle{amsplain}
\providecommand{\bysame}{\leavevmode\hbox to3em{\hrulefill}\thinspace}
\providecommand{\MR}{\relax\ifhmode\unskip\space\fi MR }
\providecommand{\MRhref}[2]{%
  \href{http://www.ams.org/mathscinet-getitem?mr=#1}{#2}
}
\providecommand{\href}[2]{#2}

\end{document}